\documentclass[11pt]{article} 

\usepackage[utf8]{inputenc} 
\usepackage[T1]{fontenc}
\usepackage[T1]{CJKutf8} 
\usepackage[english]{babel}
\usepackage{amsmath, amsthm, amssymb}
\usepackage{amssymb,amsfonts,textcomp}
\usepackage{supertabular}
\usepackage{hhline}
\usepackage{multicol}
\usepackage{array}
\usepackage{graphicx} 
\usepackage{color}
\usepackage{hyperref}
\usepackage{placeins}
\usepackage{algorithmicx}
\usepackage[ruled]{algorithm}
\usepackage{algpseudocode}
\usepackage{appendix}
\usepackage{footnote}
\usepackage{subfigure}
\usepackage{caption}
\usepackage{color, colortbl}

\theoremstyle{plain}
\newtheorem{theorem}{Theorem}

\newtheorem{notation}{Notation}

\theoremstyle{definition}
\newtheorem{definition}{Definition}
\newtheorem{example}{Example}

\newtheorem{open problem}{Open Problem}

\theoremstyle{remark}
\newtheorem{remark}[theorem]{Remark}
\newtheorem{note}[theorem]{Note}

\definecolor{Gray}{gray}{0.9}
\definecolor{LightCyan}{rgb}{0.88,1,1}
\definecolor{Red}{rgb}{1,0,0}
\definecolor{Orange}{rgb}{1,0.5,0}

\usepackage{geometry} 
\usepackage{fancyhdr} 
\usepackage{sectsty}
\allsectionsfont{\sffamily\mdseries\upshape} 

\title{Some properties of catalog of (3, g) Hamiltonian bipartite graphs: orders, non-existence and infiniteness}
\author{Vivek S. Nittoor\\
\small\tt vivek@nittoor.com\\
\small\tt Independent Consultant \& Researcher \footnote{formerly with the University of Tokyo}}
\date{} 

\begin{document}
\maketitle

\begin{abstract}

The focus of this paper is on discussion of a catalog of a class of $(3, g)$ graphs for even girth $g$. A $(k, g)$ graph is a graph with regular degree $k$ and girth $g$. This catalog is compared with other known lists of $(3, g)$ graphs such as the enumerations of trivalent symmetric graphs and enumerations of trivalent vertex-transitive graphs, to conclude that this catalog has graphs for more orders than these lists. This catalag also specifies a list of orders, rotational symmetry and girth for which the class of $(3, g)$ graphs do not exist. It is also shown that this catalog of graphs extends infintely. 
\end{abstract}

\section{Background}
\label{sec_background}
A catalog of $(3, g)$ graphs for even girth $g$ has been listed in \cite{CatalogPaper}. A $(k, g)$ graph is a graph with regular degree $k$ and girth $g$. This catalog has been listed for the Hamiltonian bipartite class of trivalent graphs. A detailed discussion on the properties of this catalog is in \cite{OverallPaper}. In this paper, we further elaborate on some of these important properties such as orders for which $(3, g)$ graphs exist on our catalog, orders for which non-existence of $(3, g)$ Hamiltonian bipartite graphs do not exist, and infinite families of $(3, g)$ Hamiltonian bipartite graphs.
A literature survey for graph catalogs and enumerations along with the properties of this catalog has been provided in \cite{OverallPaper}. The focus of this paper is as follows.
\begin{enumerate}
\item A discussion on other methods to construct graphs of high girth has been provided in section \ref{sec_other_methods}. A more detailed survey on catalogs of graphs can be found in \cite{OverallPaper}. 
\item A detailed comparion with existing works, i.e., the enumerations of trivalent symmetric graphs and trivalent vertex-transitive graphs in section \ref{sec_dense_comp}.
\item A detailed list of orders, rotational symmetry and girth for which the class of $(3, g)$ graphs do not exist has been provided in section \ref{sec_non_existence_list}.
\item Infiniteness of our catalog of $(3, g)$ graphs has been explained in section \ref{sec_inf_graph}. Our catalog of $(3, g)$ graphs extends for infinite number of orders, and hence this is an important difference from other lists of $(3, g)$ graphs which are finite.
\end{enumerate}

The cage problem refers to finding the smallest $(k, g)$ graph. Proving that a $(k,g)$ graph is a cage involves showing that a smaller $(k,g)$ graph does not exist. The cage problem is currently resolved only for very limited set of degrees $k$ and girth $g$. If a catalog of $(k, g)$ graphs lists orders for which $(k,g)$ graphs do not exist, then it would resolve the order of the $(k, g)$ cage. Table \ref{table_knowntrivalentbounds} shows the bounds for trivalent cages for even values of girth and known cages from Exoo et al. 2011 \cite{Jajcaysurvey}.\\
Further if the catalog lists at least one $(k, g)$ graph for each order for which such graphs exist, then at least one $(k, g)$ cage can be found. However, the difficulty in enumerating the entire set of $(k, g)$ graphs even for $k = 3$ is known to be very high as shown by the following two surveys on enumeration of trivalent graphs, Read 1981 \cite{Read1981} and Brinkmann et al. 2013 \cite{Brinkmann_hist2013}.\\
The $(k, g)$ cage sub-problem for the vertex-transitive graphs is resolved for many more sets of degree $k$ and girth $g$ than the more general $(k, g)$ cage problem, and also enumerations for trivalent vertex-transitive graphs are known to exist, as recently extended by $2012$, Potocnik et al. \cite{VTcen1} \cite{VTcen2}.\\
Hence, there are two important areas of research literature that are relevant for this research. One is related to the cage problem Exoo et al \cite{Jajcaysurvey} 2011, and the other is related to the enumerations of classes of trivalent graphs. Quoting from \cite{OverallPaper}, ``The enumeration of trivalent symmetric graphs began with the Foster census and was expanded by Conder et al. \cite{mconder} up to order $10,000$. In $2012$, Potocnik et al. \cite{VTcen1} \cite{VTcen2} extended the list to trivalent vertex-transitive graphs up to order $1280$. Sophisticated techniques in these developments are based upon the classification of finite simple groups. These lists provide useful data to various areas of graph theory.'' A detailed survey of the literature of enumerations of graphs is provided in \cite{OverallPaper}.\\
The difficulty of finding a $(k, g)$ graph for a particular class determines whether the $(k, g)$ cage sub-problem for that partcular class could be resolved, and also determines whether an enumeration of $(k, g)$ grapsh for that particular class is feasible. Hence, these two areas of research are related to each other. 
\begin{table}
\centering
\caption{Bounds for trivalent cages for even values of girth from \cite{Jajcaysurvey}}
\label{table_knowntrivalentbounds}
\begin{tabular}{lllllll}
\hline\noalign{\smallskip}
Girth $g$ &
$n(3,g)$ &
Number of cages &
Due to \\
\noalign{\smallskip}
\hline
\noalign{\smallskip} 
6 &
14 &
1 &
Heawood\\
8 &
30 &
1 &
 Tutte\\
10 &
70 &
3 &
O' Keefe-Wong\\
12 &
126 &
1 &
Benson\\
14 &
$258\le n(3,14)\le 384$
 &
~
 &
Exoo; (Lower Bound -McKay) \\
16 &
$512\le n(3,16)\le 960$
 &
~
 &
Exoo\\
18 &
$1024\le n(3,18)\le 2560$
 &
~
 &
Exoo\\
20 &
$2048\le n(3,20)\le 5376$
 &
~
 &
Exoo\\
22 &
$4096\le n(3,22)\le 16206$
 &
~
 &
 Biggs-Hoare\\
24 &
$8192\le n(3,24)\le 49608$
 &
~
 &
 Bray-Parker-Rowley\\
26 &
$16384\le n(3,26)\le 109200$
 &
~
 &
 Bray-Parker-Rowley\\
28 &
$32768\le n(3,28)\le 415104$
 &
~
 &
 Bray-Parker-Rowley\\
30 &
$65536\le n(3,30)\le 1143408$
 &
~
 &
 Exoo-Jajcay\\
32 &
$131072\le n(3,32)\le 3650304$
 &
~
 &
 Bray-Parker-Rowley\\
& 
 \\
\hline
\end{tabular}
\end{table}

\begin{table}
\centering
\caption{Known trivalent cages from \cite{Jajcaysurvey}}
\label{table_trivalentCages}
\begin{tabular}{lllllll}
\hline\noalign{\smallskip}
Girth $g$ &
Order $n(3,g)$ &
Number of Cages \\
\noalign{\smallskip}
\hline
\noalign{\smallskip} 
5 & 10 & 1 \\
6 & 14 & 1 \\
7 & 24 & 1 \\
8 & 30 & 1 \\
9 & 58 & 18 \\
10 & 70 & 3 \\
11 & 112 & 1 \\
12 & 126 & 1 \\
&
 \\
\hline
\end{tabular}
\end{table}

We now turn our attention to our catalog of $(3, g)$ graphs. For the convenience of the reader, quoting from \cite{OverallPaper}, ``The important steps in the approach to find graphs of high girth can be described as follows.
\begin{enumerate}
\item The search space for computer search is restricted to Hamiltonian trivalent bipartite class of trivalent graphs.
\item An efficient representation for Hamiltonian trivalent bipartite graphs with a specified level of rotational symmetry.
\item  A range of rotational symmetries are chosen wisely for each value of $g$ such that a $(3, g)$ HBG with that level of rotational symmetry could be found by computer search.
\item We treat $(3, g)$ graphs of a particular level of rotational symmetry within the identified class of trivalent graphs as a subclass, and seek to list at least one representative from each subclass.''
\end{enumerate}

The efficient representation for Hamiltonian trivalent bipartite graphs with a specified level of rotational symmetry, referred above is the D3 chord index notation from \cite{OverallPaper} which is explained at the end of this section.\\
Symmetry factor is a parameter for representing rotational symmetry in a Hamiltonian trivalent bipartite graph that has been defined in \cite{OverallPaper}. Quoting from \cite{OverallPaper}, ``Symmetry factor allows the decomposition of the problem of listing $(3, g)$ HBGs to sub-problems of listing $(3, g)$ HBGs for a range of symmetry factors.'' 
\begin{definition}  \textbf{Symmetry factor for Hamiltonian trivalent bipartite graph}  \cite{OverallPaper}\\
\label{def_sym_fac_gen}
\label{def_sym_fac}
A Hamiltonian trivalent bipartite graph with order $2m$ is said to have symmetry factor $b \in \mathbb{N}$ if the following conditions are satisfied.
\begin{enumerate}
\item $b | m$. 
\item There exists a labelling of the vertices of the Hamiltonian trivalent graph with order $2m$ have labels $1, 2, \ldots, 2m$, such that $1 \to 2 \to \ldots \to 2m \to 1$ is a Hamiltonian cycle that satisfy the following properties.
\begin{itemize}
\item The edges that are not part of the above Hamiltonian cycle are connected as follows.
Vertex $i$ is connected to vertex $u_{i}$ for $1 \le i \le 2m$.
\item If $j \equiv i \bmod 2b$ for $1 \le j \le 2b$ and $1 \le i \le 2m$ then the following is true, $u_{i} - i  \equiv u_{j} - j \bmod 2m$.
\end{itemize}
\end{enumerate}
\end{definition}

We quote the following for \nameref{notation_D3} notation from \cite{OverallPaper}.

\subsubsection{D3}
\label{notation_D3}
The notation \nameref{notation_D3} refers to \nameref{notation_D3} chord indices $l_{1}, l_{2}, \ldots, l_{b}$ for a Hamiltonian trivalent bipartite graph with symmetry factor $b$ and $2m$ vertices where $b | m$. The definition of \nameref{notation_D3} chord index notation has been 
has been reproduced here from \cite{OverallPaper}.

\begin{definition} \textbf{D3 chord index notation}   \cite{OverallPaper}\\
\label{notation_trivalent_bg_m}
The \nameref{notation_D3} chord indices $l_{1}, l_{2}, \ldots, l_{m}$ for order $2m$ where each $l_{i}$ is an odd integer satisfying $3 \le l_{i} \le 2m - 3$ for $1\le i\le m$ is a labeled graph with order $2m$, with labels $1, 2, 3, \ldots, 2m$ constructed as follows. 
\begin{enumerate}
\item Vertex $1$ is connected to vertex $2m$,  vertex $2$ and vertex $1 + l_{1}$.
\item For integers $i$ satisfying $2 \le i \le m$, vertex $2i - 1$ is connected to the following three vertices with even labels.
\begin{itemize}
\item Vertex $2i - 1$ is connected to vertex $2i - 2$.
\item Vertex $2i - 1$ is connected to vertex $2i$.
\item Vertex $2i - 1$ is connected to vertex $y_{i}$.
\end{itemize}
where $y_{i}$ is calculated as follows.
\begin{itemize}
\item $y_{i} = 2i - 1 + l_{i}$ if $2i - 1 + l_{i} \le 2m$.
\item $y_{i} = 2i - 1 + l_{i} \bmod 2m$ if $2i - 1 + l_{i} > 2m$.
\end{itemize}
\end{enumerate}
\end{definition}

\FloatBarrier

\section{Related Works}
\label{sec_other_methods}





As discussed in Section \ref{sec_background}, the cage problem is related to our research on catalog of graphs. The cage problem is related to construction of regular graphs of high girth. A related question is construction of graphs of regular degree with high girth. Many high girth graph individual construction techniques have been listed in Exoo et al \cite{Jajcaysurvey} 2011. As mentioned in Section \ref{sec_background}, a detailed survey of enumerations and lists of graphs has been provided in \cite{OverallPaper}, and hence the focus in this section is on graph constructions from the literature. Most of these authors and works have been referred to in \cite{Jajcaysurvey} as well. 


I now discuss some of the important constructions for trivalent graphs from the literature.
\paragraph{Voltage graph} \hspace{0pt} \\
The voltage graph construction method has been used by Exoo in \cite{Exoo} for the $(3, 14)$ record graph with order $384$, Exoo in \cite{113} for $(3, 16)$ record graph with order $960$,
Exoo in \cite{116} for $(3, 18)$ record graph with order $2560$, Exoo in \cite{114} for $(3, 20)$ record graph with order $5376$, and Exoo and Jajcay in \cite{116} for $(3, 30)$ record graph with order $1143408$.
\paragraph{Cayley graph}  \hspace{0pt} \\
Quoting from \cite{Jajcaysurvey}: "Bray, Parker and Rowley constructed a number of current record holders for degree three by factoring out the $3$-cycles in trivalent Cayley graphs." Several of the current trivalent record holder graphs described in \cite{BrayParkerRowley} that are constructed by this method and are $(3, 24)$ record graph of order $49608$, $(3, 26)$ record graph of order $109200$, $(3, 28)$ record graph of order $415104$, $(3, 32)$ record graph of order $3650304$.
\paragraph{Trivalent symmetric graphs}  \hspace{0pt} \\
The Foster census to enumerate all trivalent symmetric graphs was initiated by Ronald M. Foster in the $1930$s. The Foster census was published as a book, Bouwer et al 1988 \cite{FosterC} with trivalent symmetric graphs until order $512$. The Extended Foster Census until order $768$ was published in Conder et al 2002 \cite{marston}. Conder et al have extended this list to order $2048$ and more recently to order $10000$.  Conder`s list of trivalent symmetric graphs up to order $2048$ is available at the link,  \url{http://www.math.auckland.ac.nz/~conder/symmcubic2048list.txt}. 
The more recent list of trivalent symmetric graphs up to order $10000$ by Conder is available at the link, \url{http://www.math.auckland.ac.nz/~conder/symmcubic10000list.txt}.
\paragraph{Vertex-transitive graphs}  \hspace{0pt} \\
A more recent effort to enumerate trivalent vertex-transitive graphs by Primož Potočnik, Pablo Spiga and Gabriel Verret in 2012 \cite{VTcen1} and \cite{VTcen2}. This is a generalization of the enumeration of trivalent symmetric graphs, and the current method works until order $1280$. 
Quoting from \cite{VTcen1}, 
``Let $\Gamma$ be a cubic $G$-vertex-transitive graph, let $v$ be a vertex of $\Gamma$ and let $m$ be the number of orbits\footnote{Orbits are equivalence classes under the relation, $x \equiv y$ if and only if there exists $h \in H$ with $h.x = y$} of the vertex-stabiliser\footnote{For $x \in X$, the stabilizer subgroup of $x$, is the set of all elements in $H$ that are fixed-points of $x$. $H_{x} = \{h \in H|h.x = x\}$. Given group $H$ acting on a set $X$, and given $h$ in $H$ and $x$ in $X$ with $h.x = h$, then $x$ is a fixed point of $h$.} $G_{v}$ in its action on the neighbourhood ($v$). It is an easy observation that, since $\Gamma$ is $G$-vertex-transitive, $m$ is equal to the number of orbits of $G$ in its action on the arcs of $\Gamma$ (and, in particular, does not depend on the choice of $v$). 
Since $\Gamma$ is cubic, it follows that $m = {1, 2, 3}$ and there is a natural split into three cases, according to the value of $m$.''\\ 
The case $m = 1$ corresponds to that of trivalent symmetric graphs for which $|G| \le 48|V(\Gamma)|$, based on a celebrated theorem from Tutte 1947 \cite{42}, Tutte 1959 \cite{Tutte_sym}, that says that the vertex-stablizer has order of at most 48.
Quoting further from \cite{VTcen1}: "Since the order of the groups involved grows at most linearly with the order of the graphs and the groups have a particular structure [12], a computer algebra system can find all the graphs up to a certain order rather efficiently (by using the LowIndexNormalSubgroups algorithm in Magma for example)."\\
For $m = 3$, quoting from \cite{VTcen1}: "If $m = 3$, then $G_{v}$ fixes the neighbours of $v$ pointwise and, by connectedness, it is easily seen that $G_{v} = 1$. This lack of structure of the vertex-stabiliser makes it difficult to use the method that was successfully used in the arc-transitive case. On the other hand, since $G_{v} = 1$, it follows that $|G| = |V(\Gamma)| \le 1280$.'' Thus, $|G| = |V(\Gamma)| \le 1280$ for $m = 3$. Quoting further from \cite{VTcen1}: "This allows us to use the SmallGroups database in Magma to find all possibilities for G (and then for $\Gamma$)"\\ 
For $m = 2$, quoting from \cite{VTcen1}: "Therefore, in order to find all cubic G-vertex-transitive graphs with $m = 2$ up to $n$ vertices, it suffices to construct the list of all tetravalent arc-transitive graphs of order at most $n/2$."
Quoting further from \cite{VTcen1}: "This allows us to use a method similar to the one used in the cubic arc-transitive case to construct a list of all tetravalent arc-transitive graphs of order at most $640$."\\
The information from enumeration of vertex-transitive graphs from \cite{VTcen1} and \cite{VTcen2} is also available at the link, \url{http://www.matapp.unimib.it/~spiga/TableLineByLine.html}.
\paragraph{Ramanujan graphs}  \hspace{0pt} \\
Quoting from Lubotzky, Phillips and Sarnak \cite{LPS_RG}, "Ramanujan graphs $X^{p,q}$ {are } $p+1$ regular Cayley graphs of the group $\mathit{PSL}(2,\mathbb{Z}/q\mathbb{Z})$ if the Legendre symbol $(\frac{p}{q})=1$ and of $\mathit{PGL}(2,\mathbb{Z}/q\mathbb{Z})$ {if the Legendre symbol } $(\frac{p}{q})=-1$. $X^{p,q}$ is bipartite of order $n=\left|(X^{p,q})\right|=q\ast (q^{2}-1)$ and a bound on the girth is given by the equation, $g(X^{p,q})\ge 4\log _{p}(q)-\log _{p}(4)$".
\paragraph{Lazebnik}  \hspace{0pt} \\
For $q$ being a power of a prime $k\ge 3$, Lazebnik in \cite{Lazebnik} describes explicit construction of a $q$-regular bipartite graph on $v=2q^{k}$ vertices with girth $g\ge k+5$.
\paragraph{Chandran}  \hspace{0pt} \\
Chandran in \cite{118} describes a high graph construction method that constructs a graph with girth $\log(n)$ with order $n$.
\paragraph{Research on improvement of lower bound}
The research on improving lower bound for $(k, g)$ consists of proving the non-existence of a $(k, g)$ graph with a given number of vertices. This approach has been used to find the correct values of the lower bound for $n(3, 11)$ and $(4, 7)$ in \cite{112}. 
Extensive computer searches have already been used for improving lower bounds for the cage problem. 
Quoting from \cite{Jajcaysurvey}:
"Such proofs are organized by splitting the problem into a large number of subproblems, which can then be handled independently, and the work can be done in parallel on many different computers. The computation can begin by selecting a root vertex and constructing a rooted $k$-nary tree of radius $\frac{(g - 1)}{2}$. The actual computation proceeds in two phases. First, all non-isomorphic ways to add sets of $m$ edges to the tree are determined (for some experimentally determined value of $m$). This phase involves extensive isomorphism checking. The second phase is the one that is more easily distributed across a large number of computers. Each of the isomorphism classes found in the first phase becomes an independent starting point for an exhaustive search to determine whether the desired graph can be completed. Of all possible edges that could be added to the graph at this point, those that would violate the degree or girth conditions are eliminated. The order in which the remaining edges are considered is then determined by heuristics."
O' Keefe and Wong detailed case analysis: These methods are for lower bound improvement by checking and establishing a $(k, g)$ cage, and different from our approach in terms of focus.
\begin{itemize}
\item $(3, 10)$: 1980, M. O'Keefe, P.K. Wong \cite{31}.
\item $(6, 5)$: 1979, M. O'Keefe, P.K. Wong \cite{KW_6_5}.
\item $(7, 6)$: 1981, M. O'Keefe, P.K. Wong \cite{KW_7_6}.
\item Girth $5$: 1984, M. O'Keefe, P.K. Wong \cite{KW_G_5}.
The following methods are general, but are different from our approach since they are more focussed on improving the lower bound and showing non-existence and then establishing cages.
\begin{itemize}

\item $(3, 9)$: 1995, Brinkmann, \cite{Brinkmann};  $(3, 11)$:  1998, Mckay \cite{105} and $(4, 7)$:   2011 Exoo \cite{112}
\item Largest case for elimination of symmetry assumption $n(3, 9) = 58$ and $n(4, 7) = 67$: Our methods do not work for odd girth, and $(3, 8)$ is the largest case that works for full symmetry factor.
\end{itemize}
\end{itemize}






\FloatBarrier

\section{Existence results}
\label{sec_dense_comp}

Our catalog of $(3, g)$ Hamiltonian bipartite graphs has more orders than lists of $(3, g)$ symmetric graphs and $(3, g)$ vertex-transitive graphs. 
The comparison of our list of $(3, g)$ Hamiltionian bipartite graphs with lists of $(3, g)$ vertex-transitive and $(3, g)$ symmetric graphs is summarized in Table \ref{table_comp1} from \cite{OverallPaper} for even values of girth $g$ between $6$ and $14$. As shown in Table \ref{table_comp2}, our lists are exhaustive for $(3, 6)$ and $(3, 8)$ Hamiltionian bipartite graphs and partial for $(3, 10)$, $(3, 12)$ and $(3, 14)$ Hamiltionian bipartite graphs.

\begin{table}
\caption{Comparison of orders of $(3, g)$ lists from \cite{OverallPaper}}
\label{table_comp1}
\centering
\begin{tabular}{cccccc}
\hline
    $(3, g)$ & Until &Hamiltonian  & Vertex  &  Symmetric   \\
 graphs  & order & bipartite & -transitive &  \\
\noalign{\smallskip}
\hline
\noalign{\smallskip}

$(3, 6)$ & 50 & 19 & 19 & 10   \\
 $(3, 8)$ &90 & 29  & 21  & 6 \\
$(3, 10)$ & 160 & 29 & 15   & 7 \\
$(3, 12)$ & 400& 84 &26  & 16 \\
$(3, 14)$ & 1000& 164 & 35  & 11\\
\hline
\end{tabular}
\end{table}

\begin{enumerate}
\item \textbf{$(3, 6)$ Hamiltonian bipartite graphs until $50$} \\
Our enumeration of distinct orders for which $(3, 6)$ Hamiltonian bipartite graphs exist until $50$ is exhaustive. $(3, 6)$ Hamiltonian bipartite graphs exist for all even orders greater than or equal to $14$. As shown in Table \ref{table_3_6_hbg}, $(3, 6)$ Hamiltonian bipartite graphs exist for $19$ distinct orders until $50$ and in Table \ref{table_3_6_s_vt}, $(3, 6)$ vertex-transitive graphs exist for $19$ orders until $50$ and $(3, 6)$ symmetric graphs exist for $19$ orders until $50$.

\item \textbf{$(3, 8)$ Hamiltonian bipartite graphs until $90$} \\
Our enumeration of distinct orders for which $(3, 8)$ Hamiltonian bipartite graphs exist until $90$ is exhaustive, since $(3, 8)$ Hamiltonian bipartite graphs have been found to exist for all distinct orders between $30$ and $90$, with the exception of $32$, for which it is shown that a $(3, 8)$ Hamiltonian bipartite graph cannot exist in Table \ref{table_non_existence}. As shown in Table \ref{table_3_8_hbg}, $(3, 8)$ Hamiltonian bipartite graphs exist for $28$ orders until $90$ and in Table \ref{table_3_8_s_vt}, $(3, 8)$ vertex-transitive graphs exist for $21$ orders until $90$ and $(3, 8)$ symmetric graphs exist for $6$ orders until $90$. It is observed in Table \ref{table_3_8_hbg} and Table \ref{table_3_8_s_vt} that for each $(3, 8)$ vertex-transitive graph on the vertex-transitive list, there exists a $(3, 8)$ Hamiltonian bipartite graph on our list until order $90$.

\begin{remark}
$\exists$ $(3, 8)$ HBG for even orders satisfying $[34, 90] \cup \{30\}$.
\end{remark}

\item \textbf{$(3, 10)$ Hamiltonian bipartite graphs until $160$} \\
Our enumeration of distinct orders for which $(3, 10)$ Hamiltonian bipartite graphs exist until $160$ is partial, since our conclusion on existence of $(3, 10)$ Hamiltonian bipartite graphs for some orders is inconclusive. As shown in Table \ref{table_3_10_hbg}, $(3, 10)$ Hamiltonian bipartite graphs exist for $29$ orders until $160$ and in Table \ref{table_3_10_s_vt}, $(3, 10)$ vertex-transitive graphs exist for $15$ orders until $160$ and $(3, 10)$ symmetric graphs exist for $7$ orders until $160$. It is observed in Table \ref{table_3_10_hbg} and Table \ref{table_3_10_s_vt} that for each $(3, 10)$ vertex-transitive graph on the vertex-transitive list, there exists a $(3, 10)$ Hamiltonian bipartite graph on our list until order $160$. It is observed in Table \ref{table_3_10_hbg} and Table \ref{table_3_10_s_vt} that for each $(3, 10)$ vertex-transitive graph on the vertex-transitive list, there exists a $(3, 10)$ Hamiltonian bipartite graph on our list until order $160$.

\item \textbf{$(3, 12)$ Hamiltonian bipartite graphs until $400$} \\
Our enumeration of distinct orders for which $(3, 12)$ Hamiltonian bipartite graphs exist until $400$ is partial, since our conclusion on existence of $(3, 12)$ Hamiltonian bipartite graphs for some orders is inconclusive. As shown in Table \ref{table_3_12_hbg}, $(3, 12)$ Hamiltonian bipartite graphs exist for $84$ orders until $400$ and in Table \ref{table_3_12_s_vt}, $(3, 12)$ vertex-transitive graphs exist for $26$ orders until $400$ and $(3, 12)$ symmetric graphs exist for $16$ orders until $400$. It is observed in Table \ref{table_3_12_hbg} and Table \ref{table_3_12_s_vt} that for each $(3, 12)$ vertex-transitive graph on the vertex-transitive list, there exists a $(3, 12)$ Hamiltonian bipartite graph on our list until order $400$. It is observed in Table \ref{table_3_12_hbg} and Table \ref{table_3_12_s_vt} that for each $(3, 12)$ vertex-transitive graph on the vertex-transitive list, there exists a $(3, 12)$ Hamiltonian bipartite graph on our list until order $400$. It is observed in Table \ref{table_3_12_hbg} and Table \ref{table_3_12_s_vt} that for each $(3, 12)$ vertex-transitive graph on the vertex-transitive list, there exists a $(3, 12)$ Hamiltonian bipartite graph on our list until order $400$.

\item \textbf{$(3, 14)$ Hamiltonian bipartite graphs until $1000$} \\
Our enumeration of distinct orders for which $(3, 14)$ Hamiltonian bipartite graphs exist until $1000$ is partial, since our conclusion on existence of $(3, 14)$ Hamiltonian bipartite graphs for some orders is inconclusive. As shown in Table \ref{table_3_14_hbg}, $(3, 14)$ Hamiltonian bipartite graphs exist for $84$ orders until $400$ and in Table \ref{table_3_14_s_vt}, $(3, 14)$ vertex-transitive graphs exist for $35$ orders until $1000$ and $(3, 14)$ symmetric graphs exist for $11$ orders until $400$. It is observed in Table \ref{table_3_14_hbg} and Table \ref{table_3_14_s_vt} that for each $(3, 14)$ vertex-transitive graph on the vertex-transitive list, there exists a $(3, 14)$ Hamiltonian bipartite graph on our list until order $1000$. 

\end{enumerate}

\begin{note} \textbf{Observation}
$(3, g)$ Hamiltonian bipartite graphs exist for each distinct order for which $(3, g)$ vertex-transitive graphs exist for considered ranges, whether bipartite or non-bipartite exist for even girth until $g =14$.
\end{note}

Quoting from \cite{OverallPaper} on outcomes of listing of $(3, g)$ Hamiltonian bipartite graphs as follows.
\begin{itemize}
\item \textit{Exhaustive}: Outcome of listing of $(3, g)$ Hamiltonian bipartite graphs is exhaustive to the extent that all orders in specified range that have a $(3, g)$ Hamiltonian bipartite graph are listed, with proof for non-existence for orders not listed.
\item \textit{Partial}: Outcome of listing of $(3, g)$ Hamiltonian bipartite graphs is partial if results on existence $(3, g)$ Hamiltonian bipartite graph for some orders in specified range are inconclusive.
\end{itemize}

\begin{table}
\caption{Catalog of $(3, g)$ Hamiltonian bipartite graphs}
\label{table_comp2}
\centering
\begin{tabular}{cllllll}
\hline
    $(3, g)$ & Until & Coverage  & Table & Upper bound \\
\noalign{\smallskip}
\hline
\noalign{\smallskip}

$(3, 6)$ & 50 &\textit{Exhaustive}   \cellcolor{Gray}   & Table \ref{table_3_6_hbg}  &$(3, 6)$ cage included  \\
 $(3, 8)$ & 90 & \textit{Exhaustive}   \cellcolor{Gray}  & Table \ref{table_3_8_hbg} & $(3, 8)$ cage included \\
$(3, 10)$ & 160 &\textit{Partial}
   \cellcolor{LightCyan}   & Table \ref{table_3_10_hbg}  &  $(3, 10)$ cage included \\
$(3, 12)$ & 400 &\textit{Partial} \cellcolor{LightCyan} &  Table \ref{table_3_12_hbg} & $(3, 12)$ cage included \\
$(3, 14)$ & 1000 &\textit{Partial}  \cellcolor{LightCyan}  & Table \ref{table_3_14_hbg} &  $(3, 14)$ record graph included  \\
\hline
\end{tabular}
\end{table}


\begin{table}
\caption{$(3, 6)$ lists}
\label{table_3_6_hbg} 
\centering
\begin{tabular}{cllllll}
\hline\noalign{\smallskip}
Order of $(3, 6)$ Hamiltonian bipartite graph \\  
\noalign{\smallskip}
\hline
\noalign{\smallskip} 

14, 16, 18, 20, 22, 24, 26, 28, 30, 32, 34, 36, 38, 40, 42, 44, 46, 48, 50\\ 
\hline
\end{tabular}
\end{table}

\begin{table}
\caption{Other $(3, 6)$ lists}
\label{table_3_6_s_vt} 
\centering
\begin{tabular}{cllllll}
\hline\noalign{\smallskip}
Class of $(3, 6)$ graph & Orders for specified class of $(3, 6)$ graph \\  
\noalign{\smallskip}
\hline
\noalign{\smallskip} 

 \textbf{$(3, 6)$ symmetric} & 14, 16, 18, 20, 24, 26, 32, 38, 42, 50  \\
 \textbf{$(3, 6)$ vertex-transitive} & 14, 16, 18, 20, 22, 24, 26, 28, 30, 32, 34, 36, 38, 40, 42, 44, 46, 48, 50\\
\hline
\end{tabular}
\end{table}



\FloatBarrier




\begin{table}
\centering
\begin{tabular}{ccc}
\hline
   & \\
\noalign{\smallskip}
\hline
\noalign{\smallskip}
\label{table_comp1}
\cellcolor[rgb]{1, 0.5, 0} & \cellcolor[rgb]{1, 0.5, 0} & Trivalent symmetric\\ 
\cellcolor[rgb]{1, 0.5, 0} & \cellcolor[rgb]{1, 0.5, 0}  &  and also on our $(3, g)$ HBG catalog \\
\hline
\cellcolor[rgb]{0, 0, 1} &\cellcolor[rgb]{0, 0, 1} &Trivalent vertex-transitive\\ 
\cellcolor[rgb]{0, 0, 1} &\cellcolor[rgb]{0, 0, 1}  & and also on our $(3, g)$ HBG catalog \\
\hline
\cellcolor[rgb]{1, 0, 0} & \cellcolor[rgb]{1, 0, 0} & Trivalent vertex-transitive\\ 
\cellcolor[rgb]{1, 0, 0} & \cellcolor[rgb]{1, 0, 0}  &  \textbf{but not on} our $(3, g)$ HBG catalog \\
\hline
\cellcolor{black} & \cellcolor{black}  & On our $(3, g)$ HBG catalog\\ 
\cellcolor{black} & \cellcolor{black}  & \textbf{but not on} trivalent vertex-transitive list \\
\hline
\cellcolor{Gray} & \cellcolor{Gray} & On our $(3, g)$ HBG catalog\\ 
\cellcolor{Gray} & \cellcolor{Gray} &  \\
\hline
\end{tabular}
\end{table}

\begin{table}
\centering
\caption{$(3, 8)$ until order 90} 
\begin{tabular}{ccccc}
\cellcolor[rgb]{1, 0.5, 0} & \cellcolor[rgb]{0, 0, 1} & \cellcolor{black} & \cellcolor[rgb]{1, 0, 0} &  \cellcolor{Gray}\\
 6 & 21 & 8 & 0 & 29\\
\hline
\end{tabular}
\begin{tabular}{c}
\cellcolor{Gray}
\textcolor[rgb]{1, 0.5, 0}{30}, {34, 36, 38}, \textcolor[rgb]{1, 0.5, 0}{40}, \textcolor[rgb]{0, 0, 1}{42}, {44}, \textcolor[rgb]{0, 0, 1}{48, 50, 52, 54}, \textcolor[rgb]{1, 0.5, 0}{56}, \textcolor[rgb]{0, 0, 1}{58, 60}, {62}, \\
\cellcolor{Gray}
 \textcolor[rgb]{1, 0.5, 0}{64}, \textcolor[rgb]{0, 0, 1}{66, 68, 70, 72, 74}, {76}, \textcolor[rgb]{0, 0, 1}{78, 80, 82, 84}, {86, 88}, \textcolor[rgb]{0, 0, 1}{90}\\ 
\end{tabular}
\end{table}

\begin{table}
\centering
\caption{$(3, 10)$ until order 160}  
\begin{tabular}{ccccc}
\cellcolor[rgb]{1, 0.5, 0} & \cellcolor[rgb]{0, 0, 1} & \cellcolor{black} & \cellcolor[rgb]{1, 0, 0} &  \cellcolor{Gray}\\
 7 & 15 & 14 & 0 & 29\\
\hline
\end{tabular}
\begin{tabular}{c}
\cellcolor{Gray}
{70, 72, 78}, \textcolor[rgb]{1, 0.5, 0}{80}, {84, 88}, \textcolor[rgb]{1, 0.5, 0}{90, 96}, {98}, \textcolor[rgb]{0, 0, 1}{100}, {104, 108}, \textcolor[rgb]{1, 0.5, 0}{110, 112, 120}, \\
 \cellcolor{Gray}
 \textcolor[rgb]{0, 0, 1}{126, 128}, {130}, \textcolor[rgb]{0, 0, 1}{132}, 136, \textcolor[rgb]{0, 0, 1}{140}, \textcolor[rgb]{1, 0.5, 0}{144}, {150, 152, 154}, \textcolor[rgb]{0, 0, 1}{156, 160}\\
\end{tabular}
\end{table}

\begin{table}
\centering
\caption{$(3, 12)$ until order 400} 
\begin{tabular}{ccccc}
\cellcolor[rgb]{1, 0.5, 0} & \cellcolor[rgb]{0, 0, 1} & \cellcolor{black} & \cellcolor[rgb]{1, 0, 0} &  \cellcolor{Gray}\\
 16 & 26 & 58 & 0 & 84\\
\hline
\end{tabular}
\begin{tabular}{c}
\cellcolor{Gray}
{126},  \textcolor[rgb]{1, 0.5, 0}{162, 168}, {180},  \textcolor[rgb]{1, 0.5, 0}{182}, {186, 190},  \textcolor[rgb]{1, 0.5, 0}{192}, {196, 198, 200},  \textcolor[rgb]{1, 0.5, 0}{204}, \\
\cellcolor{Gray}
{208, 210},  \textcolor[rgb]{1, 0.5, 0}{216}, {220, 222},  \textcolor[rgb]{1, 0.5, 0}{224}, {228, 230, 232},  \textcolor[rgb]{1, 0.5, 0}{234}, {238},  \textcolor[rgb]{1, 0.5, 0}{240}, \\
\cellcolor{Gray}
\textcolor[rgb]{0, 0, 1}{248, 250}, {252},  \textcolor[rgb]{0, 0, 1}{256}, {260},  \textcolor[rgb]{0, 0, 1}{264}, {266},  \textcolor[rgb]{0, 0, 1}{270, 272}, {276, 280, 282}, \\
\cellcolor{Gray}
 \textcolor[rgb]{0, 0, 1}{288}, {290, 294, 300, 304, 306, 308, 310},  \textcolor[rgb]{0, 0, 1}{312}, {318},  \textcolor[rgb]{0, 0, 1}{320}, {322},  \\
 \cellcolor{Gray}
 \textcolor[rgb]{0, 0, 1}{324}, {328},  \textcolor[rgb]{0, 0, 1}{330, 336}, {340},  \textcolor[rgb]{0, 0, 1}{342}, {344, 348, 350, 352, 354},  \textcolor[rgb]{0, 0, 1}{360}, \\
 
\cellcolor{Gray}
 \textcolor[rgb]{0, 0, 1}{364}, {366, 368, 370, 372, 374, 376},  \textcolor[rgb]{0, 0, 1}{378},{380},  \textcolor[rgb]{1, 0.5, 0}{384}, \textcolor[rgb]{0, 0, 1}{390}, {392},\\ 
 
 \cellcolor{Gray}
{396, 400} \\ 
\end{tabular}
\end{table}

\begin{table}
\centering
\caption{$(3, 14)$ until order 1000} 
\begin{tabular}{ccccc}
\cellcolor[rgb]{1, 0.5, 0} & \cellcolor[rgb]{0, 0, 1} & \cellcolor{black} & \cellcolor[rgb]{1, 0, 0} &  \cellcolor{Gray}\\
 11 & 35 & 129 & 1\footnote{\textcolor[rgb]{1, 0, 0}{486}} & 164\\
\hline
\end{tabular}
\begin{tabular}{c}
\cellcolor{Gray}
{{384\footnote{$(3, 14)$ record}}, \textcolor[rgb]{0, 0, 1}{406}, {440}, \textcolor[rgb]{1, 0.5, 0}{448}, {456, 460, 462, 464, 468, 472}, {476}, \textcolor[rgb]{0, 0, 1}{480}, 488, }\\

\cellcolor{Gray} 
{{490, 492, 496, 500}, \textcolor[rgb]{1, 0.5, 0}{504, 506}, {510}, \textcolor[rgb]{0, 0, 1}{512}, 516, 518, 520, 522, 528, }\\

\cellcolor{Gray}  
{{530, 532, 536, 540}, \textcolor[rgb]{0, 0, 1}{544}, 546, 550, 552, 558, 560, 564, 568, 570, } \\

\cellcolor{Gray}
{{572, 574}, \textcolor[rgb]{0, 0, 1}{576}, {580, 584, 588, 590, 592, 594}, \textcolor[rgb]{0, 0, 1}{600, 602}, 608, 610,} \\

\cellcolor{Gray}
{{612, 616, 620}, \textcolor[rgb]{1, 0.5, 0}{624}, {630, 632, 636, 638},} {\textcolor[rgb]{0, 0, 1}{640}, {644}, \textcolor[rgb]{0, 0, 1}{648}, {650, 656},}\\

\cellcolor{Gray}
{658, \textcolor[rgb]{0, 0, 1}{660}, {664, 666, 670}, \textcolor[rgb]{1, 0.5, 0}{672}, {680, 682, 684, 686, 688, 690, 696,}} \\

\cellcolor{Gray}
{700, 702, 704, 708, 710, 712}, {\textcolor[rgb]{0, 0, 1}{720}, {728, 730, 732, 736, 738, 740},} \\

\cellcolor{Gray}
744, \textcolor[rgb]{0, 0, 1}{750}, {752, 756, 760}, \textcolor[rgb]{1, 0.5, 0}{768}, {770, 774, 776, 780}, \textcolor[rgb]{0, 0, 1}{784}, {790, 792}, \\

\cellcolor{Gray}
798, \textcolor[rgb]{0, 0, 1}{800}, {808, 810}, {812}, \textcolor[rgb]{1, 0.5, 0}{816}, \textcolor[rgb]{0, 0, 1}{820}, {824, 826, 828, 830, 832}, \textcolor[rgb]{0, 0, 1}{840},\\

\cellcolor{Gray} 
{846, 848, 850, 854, 856, 860, 864, 868, 870, 872, 876}, \textcolor[rgb]{1, 0.5, 0}{880},  {882}, \\

\cellcolor{Gray}
{884, 888, 890}, \textcolor[rgb]{1, 0.5, 0}{896},  \textcolor[rgb]{0, 0, 1}{900},  {904, 910}, \textcolor[rgb]{1, 0.5, 0}{912}, {918, 920, 924, 928},  \textcolor[rgb]{0, 0, 1}{930},\\

\cellcolor{Gray}  
 {936, 938, 940, 944, 948, 950, 952, 954}, \textcolor[rgb]{1, 0.5, 0}{960}, {962, 968, 970}, \textcolor[rgb]{0, 0, 1}{972},\\

\cellcolor{Gray}
{976, 980}, \textcolor[rgb]{0, 0, 1}{984}, {988}, \textcolor[rgb]{0, 0, 1}{990}, \textcolor[rgb]{0, 0, 1}{992}, \textcolor[rgb]{0, 0, 1}{994}, {996}, \textcolor[rgb]{0, 0, 1}{1000}

\end{tabular}
\end{table}

\FloatBarrier

\section{Non-existence Lists}
\label{sec_non_existence_list}

The cases for which a conclusive result has been reached for non-existence of a graph with a specified symmetry factor and girth are referred to as "Non-existence List".

\begin{enumerate}
\item There does not exist a $(3, 14)$ Hamiltonian bipartite graph with symmetry factors $4, 5, 6$ between orders $258$ and $384$.
\item Non-existence of orders of $(3, 6)$ and $(3, 8)$ Hamiltonian bipartite graphs are given in Table \ref{table_non_existence} for full symmetry factors. 
\item Non-existence of orders of $(3, 8)$ Hamiltonian bipartite graphs for various symmetry factors, are given in Table \ref{table_non_existence1}.
\item Non-existence of orders of $(3, 10)$ Hamiltonian bipartite graphs for various symmetry factors, are given in Table \ref{table_non_existence2}.
\item Non-existence of orders of $(3, 12)$ Hamiltonian bipartite graphs for various symmetry factors, are given in Table \ref{table_non_existence3}. 
\item Non-existence of orders of $(3, 14)$ Hamiltonian bipartite graphs for various symmetry factors, are given in Table \ref{table_non_existence4}.
\item Non-existence of orders of $(3, 16)$ Hamiltonian bipartite graphs for various symmetry factors, are given in Table \ref{table_non_existence5}.
\item Non-existence of orders of $(3, 18)$ Hamiltonian bipartite graphs for various symmetry factors, are given in Table \ref{table_non_existence6}.
\end{enumerate}

\FloatBarrier
\begin{table}
\caption{Non-existence of $(3, g)$ Hamiltonian bipartite graphs for the following number of even vertices}
\centering
\begin{tabular}{clclllll}
\hline\noalign{\smallskip}
$(3, g)$ & Orders for non-existence of $(3, g)$ Hamiltonian bipartite graph\\  
\noalign{\smallskip}
\hline
\noalign{\smallskip} 
\label{table_non_existence} 
(3, 6) & 10, 12 \\  
(3, 8) & 20, 22, 24, 26, 28, 32 \\  
\hline
\end{tabular}
\end{table}


\begin{table}
\centering
\caption{Non-existence of $(3, 8)$ Hamiltonian bipartite graphs for various symmetry factors for the following orders}
\label{table_non_existence1} 
\begin{tabular}{cclclllll}
\hline\noalign{\smallskip}
Symmetry factor $b$ & Orders for non-existence of $(3, 8)$ \\ 
 & Hamiltonian bipartite graph\\
 & with symmetry factor $b$ \\
\noalign{\smallskip}
\hline
\noalign{\smallskip} 

 3 & 36 \\  
 4 & 32\\  
 5 & 30  \\  
 6 & 24 \\  
 7 & 28 \\  
 8 & 32 \\  
 9 & 36 \\  
 $10$ &  20 \\  
\hline
\end{tabular}
\end{table}

\begin{table}
\centering
\caption{Non-existence of $(3, 10)$ Hamiltonian bipartite graphs for various symmetry factors for the following orders}
\label{table_non_existence2} 
\begin{tabular}{clclllll}
\hline\noalign{\smallskip}
Symmetry factor $b$ & Orders for non-existence of $(3, 10)$ \\ 
 & Hamiltonian bipartite graph\\
 & with symmetry factor $b$ \\
\noalign{\smallskip}
\hline
\noalign{\smallskip} 

 $4$ & 64 \\  
 $5$ & 50, 60, 70 \\  
 $6$ &  60, 72 \\  
 $7$ & 56, 84 \\  
 $8$ & 64 \\  
 $9$ & 54, 72 \\   
 $10$ & 60 \\  
 $12$ & 24, 48 \\   
\hline
\end{tabular}
\end{table}

\begin{table}
\centering
\caption{Non-existence of $(3, 12)$ Hamiltonian bipartite graphs for various symmetry factors for the following orders}
\label{table_non_existence3} 
\begin{tabular}{clclllll}
\hline\noalign{\smallskip}
Symmetry factor $b$ & Orders for non-existence of $(3, 12)$ Hamiltonian bipartite graph\\
 & with symmetry factor $b$ \\
\noalign{\smallskip}
\hline
\noalign{\smallskip} 
 $2$ & 60 -- 512, in steps of 4  \\   
 $3$ & 132, 138, 144, 150, 156, 168, 174  \\   
 $4$ & 56 -- 208, in steps of 8  \\   
 $5$ & 60 -- 180, in steps of 10, 200  \\   
 $6$ & 60, 72, 84, 96, 108, 120, 132, 144, 156  \\   
 $7$ & 140, 154, 168   \\   
 $8$ & 128  \\   
$9$ & 144  \\ 
\hline
\end{tabular}
\end{table}

\begin{table}
\centering
\caption{Non-existence of $(3, 14)$ Hamiltonian bipartite graphs for various symmetry factors for the following orders from \cite{3_14Paper}}
\label{table_non_existence4} 
\begin{tabular}{clclllll}
\hline\noalign{\smallskip}
Symmetry factor $b$ & Orders for non-existence of $(3, 14)$ Hamiltonian bipartite graph\\
 & with symmetry factor $b$ \\
\noalign{\smallskip}
\hline
\noalign{\smallskip} 

$4$ & 272, 280, 288, 296, 304, 312, 320, 328, 336, 344, 352, 360, 368, 376, 384,   \\  
 & 392, 400, 408, 416, 424, 432, 456 \\  
$5$ & 260, 270, 280, 290, 300, 310, 320, 330, 340, 350, 360, 370, 380, 390, 400, \\   
 & 410, 420, 430, 440, 450, 470, 480 \\   
$6$ & 264, 276, 288, 300, 312, 324, 336, 348, 360, 372, 384\\ & 396, 408, 420, 432, 444  \\  
$7$ &266, 280, 294, 308, 322, 336, 350 \\  
$8$ & 272 \\  
$9$ &  270 \\
$19$ & 380 \\
\hline
\end{tabular}
\end{table}

\begin{table}
\centering
\caption{Non-existence of $(3, 16)$ Hamiltonian bipartite graphs for various symmetry factors for the following orders}
\label{table_non_existence5} 
\begin{tabular}{lllll}
\hline\noalign{\smallskip}
Symmetry factor $b$ & Orders for non-existence of $(3, 16)$ Hamiltonian bipartite graph\\
 & with symmetry factor $b$ \\
\noalign{\smallskip}
\hline
\noalign{\smallskip} 
$5$ & 950 \\   
\hline
\end{tabular}
\end{table}

\begin{table}
\centering
\caption{Non-existence of $(3, 18)$ Hamiltonian bipartite graphs for various symmetry factors for the following orders}
\label{table_non_existence6} 
\begin{tabular}{lllll}
\hline\noalign{\smallskip}
Symmetry factor & $(3, 18)$ Non-existence \\  
\noalign{\smallskip}
\hline
\noalign{\smallskip} 

$4$ & 1920 \\   
\hline
\end{tabular}
\end{table}

\FloatBarrier

\section{Infinite family of graphs}
\label{sec_inf_graph}
The \nameref{notation_D3} chord index notation can specify an infinite family of graphs.
We quote two examples from \cite{OverallPaper} \\
\begin{example} \cite{OverallPaper} \\
\nameref{notation_D3} chord index 5 leads to a $(3, 6)$ HBG for all even orders greater than or equal to $14$.
\end{example}
\begin{example} \cite{OverallPaper} \\
It is practically observed that \nameref{notation_D3} chord indices \\$15\ 53\ 73\ 139\ 243\ 267\ 471\ 651$ leads to $(3, 16)$ HBGs for orders $2352 + 16i$ for integers $i \ge 0$, for symmetry factor $8$. In addition, it is also observed that the above mentioned \nameref{notation_D3} chord indices also lead to $(3, 16)$ HBGs for the following orders \\$1824\ 1840\ 1936\ 2016\ 2032\ 2112\ 2144\ 2160\ 2176\ 2240\ 2256\ 2272$\\ $2288\, 2304\, 2320$.
\end{example}

\begin{notation} 
\textbf{c(x)} \\
Given any node with label $x$ where $1 \le x \le 2m$, we denote the node to which the chord connects to node $x$ as $c(x)$.\\
If the HBG has D3 chord indices $d_1. d_2, ...., d_b$, symmetry factor $b$ and order $2m$, then 
if $x$ is odd, then $c(x)$ is defined as follows.\\
$i = x\ \%\ 2b$ \\
if($i = 0$) then $i = 2b$ \\
$c(x) = (h + d_{i})\ \%\ 2m$ \\
if $c(x) = 0$ then $c(x) = 2m$\\
If $x$ is even, then $c(x)$ is defined as an odd number $y$, where $c(y) = x$.
\end{notation}

\begin{notation} 
\textbf{p(x)} \\
Given any node with label $x$ where $1 \le x \le 2m$, we denote the previous node p(x) as 
$p(x) = (x - 1) \ \%\ 2m$ \\
if $p(x) = 0$ then $p(x) = 2m$
\end{notation}

\begin{notation} 
\textbf{n(x)} \\
Given any node with label $x$ where $1 \le x \le 2m$, we denote the next node n(x) as 
$n(x) = (x + 1) \ \%\ 2m$ \\
if $p(x) = 0$ then $p(x) = 2m$
\end{notation}

\begin{note} 
\label{t_hbg}
\textbf{Traversals on HBGs} \\
Let us examine breadth first traversals for a HBG with D3 chord indices $d_1. d_2, ...., d_b$, symmetry factor $b$ and order $2m$.\\
Without loss of generality we consider breadth first traversals starting at nodes $1, 2, ...., 2b$ (due to rotational symmetry). Let us consider all traversals starting at node $h$ where $1 \le h \le 2b$, as tree with the following structure.\\
Starting from depth $0$, where we have one node with label $h$, we visualize a tree using the following rules.
\begin{itemize}
\item Each node $x$ has successor nodes $c(x)$ if parent node $y$ of $x$ then $x \ne c(y)$, and $p(x)$ if parent node of $x$ is not $n(x)$ and $n(x)$ if parent node of $x$ is not $p(x)$.
\item \textbf{Repeated label criterion} 
If any of the nodes at depth $k$ have the same label as another other node of the same tree at depth $k_1$ where $k_1 \le k$.
\end{itemize}
Since for simple graphs, there exist at most one edge between any two nodes, and no edge should be repeated in any traversal, each breadth first traversal tree for HBG has three nodes at depth $1$, but after that each node has exactly two child nodes, since the immediate parent node is omitted, since the same edge cannot be repeated in this traversal.
\end{note}

\begin{notation} 
\textbf{$T(h, b, m, (d_1, d_2, ...., d_b))$} for a HBG with D3 chord indices $d_1, d_2, ...., d_b$, symmetry factor $b$ and order $2m$ \\
We denote the breadth first traversals for a HBG with D3 chord indices $d_1. d_2, ...., d_b$, symmetry factor $b$ and order $2m$ as described in Note \ref{t_hbg} as $T(h, b, m, (d_1, d_2, ...., d_b))$.
\end{notation}

\begin{notation} Starting depth and terminating depth for a label $a$ in a HBG with D3 chord indices $d_1. d_2, ...., d_b$, symmetry factor $b$ and order $2m$\\
if a label $a$ is found at first at depth $s_a$ and next at depth $t_a$, in a breadth traversal tree starting from node $h$, $T(h, b, m, (d_1, d_2, ...., d_b))$ where $s_a \le t_a$, then we refer to
$s_a$ as starting depth for node $a$ in $T(h, b, m, (d_1, d_2, ...., d_b))$ and 
$t_a$ as starting depth for node $a$ in $T(h, b, m, (d_1, d_2, ...., d_b))$.\\
Lenght of the cycle is clearly $s_a + t_a$.
\end{notation}

Since all cycles for HBGs are of even lenght, any cycle can be found on one of the $2b$ breadth first traversal trees $T(h, b, m, (d_1, d_2, ...., d_b))$ such that $s_a = t_a$.



\begin{theorem}

The length of the cycle terminated at the lowest depth in the $2b$ traversal trees from node $h$ where $1 \le h \le 2b$ is in fact the girth of the HBG.

\end{theorem}
\begin{proof}
Let us consider $2b$ breadth first traversal trees starting at node $h$, $T(h, b, m, (d_1, d_2, ...., d_b))$ where $1 \le h \le 2b$.\\
Let the cycle $C_a$ at terminated at lowest depth $t_a$ for label $a$ and starting depth $s_a$ in $T(h, b, m, (d_1, d_2, ...., d_b))$ where $1 \le h \le 2b$. Clearly $s_a = t_a$, since the cycle has been terminated at the lowest depth.\\
Case A: Let us consider a cycle $C_c$ terminated at depth $t_c$ for label $c$ and starting depth $s_c$ in $T(h_1, b, m, (d_1, d_2, ...., d_b))$ where $1 \le h_1 \le 2b$ such that $s_c \ge s_a$ and $t_c \ge t_a$. Therefore $s_a + t_a \le s_c + t_c$. Therefore, cycle $C_c$ is not of smaller length than $C_a$.
Case B: 
Let us consider a cycle $C_g$ terminated at depth $t_g$ for label $g$ and starting depth $s_g$ in $T(h_1, b, m, (d_1, d_2, ...., d_b))$ where $1 \le h_1 \le 2b$ such that $s_g > s_a$ and $t_g < t_a$ such that $s_g + t_g < s_a + t_a$.\\
Let $g_1 = g \ \%\ 2b$. If $g_1$ equals $0$ then $g_1 = 2b$. Since clearly, $1 \le g_1 \le 2b$. Consider $T(g_1, b, m, (d_1, d_2, ...., d_b))$ and clearly, label $g_1$ must appear at depth $s_g + t_g$.\\
Since the cycle length is even for bipartite graphs, $s_g + t_g$ is even and if we consider $k = (s_g + t_g)/2$. let us consider $k_1 = k\ \%\ 2b$. If $k_1$ equals $0$, then $k_1 = 2b$.\\
Let us consider breadth first traversal tree $T(k_1, b, m, (d_1, d_2, ...., d_b))$. Clearly this must have repeated labels at depth $k_1$ with label $t$ such that $t\ \%\ 2b = k_1\  \%\ 2b$.\\
This means that this cycle $C_g$ terminates at a depth $k_1$ in $T(k_1, b, m, (d_1, d_2, ...., d_b))$.
Since cycle $C_g$ terminates at a lower depth compared to cycle $C_a$, we have a contradiction since 
Proof by contradiction. QED.
\end{proof}





\begin{theorem}
Given D3 chord indices $d_1, d_2, ...., d_b$ with symmetry factor $b$, there exists a threshold $2bk$, such that HBGs with D3 chord indices $d_1. d_2, ...., d_b$ with symmetry factor $b$ and order $2b (k + i)$ have the same girth, where $i = 0, 1, 2, ...$.
\end{theorem}
\begin{proof}
Let us consider a HBG with order $2m_1$, symmetry factor $b$, with D3 chord indices $d_1. d_2, ...., d_b$. Let us consider the Breadth first traversal trees $T(h, b, m_1, (d_1, d_2, ...., d_b))$ starting at nodes $h$ such that $1 \le h \le 2b$ as a function of order $2m_1$.\\
Similarly, let us consider a HBG with order $2m_2$, symmetry factor $b$, with D3 chord indices $d_1. d_2, ...., d_b$. Let us consider the Breadth first traversal trees $T(h, b, m_2, (d_1, d_2, ...., d_b))$ starting at nodes $h$ such that $1 \le h \le 2b$ as a function of order $2m_2$.\\
It is clear that we can establish a one-one onto map between $T(h, b, m_1, (d_1, d_2, ...., d_b))$ and $T(h, b, m_2, (d_1, d_2, ...., d_b))$, as long as both $2m_1$ and $2m_2$ are beyond the threshold such that that all node labels in tree are distinct, until the first repeated label corresponding to a cycle is found, which is turn means that the girth of the HBGs beyond the threshold is the same.

\end{proof}

In examples \ref{ex_1}, and \ref{ex_2}, we consider HBGs with symmetry factor $1$, D3 chord index $5$ and orders 12, and 14 respectively, and from the catalog, we know the girth for example \ref{ex_1} is $4$ and that for example \ref{ex_2} is $6$.\\
It is also known from the catalog \cite{CatalogPaper} that the more general case of example \ref{ex_3}, HBGs with symmetry factor $1$, D3 chord index $5$ and order $2m$ where $2m \ge 14$ have girth $6$.

\begin{example}
\label{ex_1}
For simplicity, let us consider symmetry factor $1$, D3 chord index $5$, and order 12.\\
Depth 1: 1 has successors 12, 2, and 6 \\
Depth 2: 12 has sucessors 11 and 7\\
2 has sucessors 3 and 9\\
6 has sucessors 7, 5\\ We stop here since 7 is a repeated label.
We now find a cycle $7 \to 6 \to 1 \to 12 \to 7$ which is of length $4$.
\end{example}

\begin{example}
\label{ex_2}
For simplicity, let us consider symmetry factor $1$, D3 chord index $5$, and order 14.\\
Depth 1: 1 has successors 14, 2, and 6 \\
Depth 2: 14 has sucessors 13 and 9\\
2 has sucessors 3 and 11\\
6 has sucessors 7, 5\\
Depth 3: 13 has successors 12 and 4\\
7 has successors 8, 12\\ We stop here since 12 is a repeated label.
We now find a cycle $12 \to 7 \to 6 \to 1 \to 14 \to 13 \to 12$ which is of length $6$.
\end{example}

\begin{example}
\label{ex_3}
 Generalizing the above idea, for a HBG symmetry factor $1$, D3 chord index $l = 5$ with an arbitrary order $2m$ where $2m \ge 14$,
 Depth 1: 1 has successors $2m$, 2, and $l + 1$ \\
Depth 2: $2m$ has sucessors $2m - 1$ and $2m - l$\\
2 has sucessors 3 and $2m - 3$\\
$l + 1$ has sucessors $l + 2$, $l$\\
Depth 3: $2m - 1$ has successors $2m - 2$ and $2m - 1 + l$\\
$l + 2$ has successors $l + 3$, $2 m - 2$\\ We stop here since $2 m - 2$ is a repeated label.
We now find a cycle $2m - 2 \to 2m-7 \to 2m - 8 \to 1 \to 2m \to 2m - 1 \to 2m - 2$ which is of length $6$.
\end{example}



\label{sec_conclusion_graph_analysis}
\label{chap_14_16_conclu}
\label{chap_6_8_10_12_conclu}

\FloatBarrier

\bibliographystyle{unsrt} 
\bibliography{vivek}

\end{document}